\documentclass[11pt,a4paper]{article}
\usepackage[a4paper,
            bindingoffset=0.2in,
            left=1in,
            right=1in,
            top=1in,
            bottom=1in,
            footskip=.25in]{geometry}
\usepackage{amsmath}
\usepackage{amsfonts}
\usepackage{graphicx}
\usepackage{cite}
\usepackage{amsthm}
\usepackage{epstopdf}
\usepackage{float}
\usepackage{authblk}
\usepackage{lipsum}
\usepackage{amsfonts}
\usepackage{graphicx}
\usepackage{epstopdf}
\usepackage{algorithmic}
\usepackage[utf8]{inputenc}
\usepackage{amsmath, amssymb}
\usepackage{geometry}
\usepackage{tikz}
\usepackage{pgfplots}
\usepackage{url}

\newtheorem{theorem}{Theorem}
\newtheorem{example}{Example}
\newtheorem{definition}{Definition}
\newtheorem{remark}{Remark}

\title{A Unified Regularity Condition for Optimal Control: \\ Bridging LICQ, MFCQ, and Subdifferentials}
\author[1]{M.~E. Abbasov\thanks{abbasov.majid@gmail.com, m.abbasov@spbu.ru}}
\affil[1]{St. Petersburg State University, SPbSU, 7/9 Universitetskaya nab., St. Petersburg, 199034 Russia}

\begin{document}

\maketitle

\begin{abstract}
This paper presents a unified derivation of transversality conditions in 
optimal control problems using exact penalty functions. The key regularity 
condition is that the origin is uniformly separated from the subdifferential 
of the penalty function in a neighborhood of the admissible set. This 
condition, hereafter referred to as the Unified Separation Condition (USC), 
generalizes the classical Mangasarian--Fromovitz condition for inequalities 
and linear independence of gradients for equalities: in the smooth case, 
LICQ and MFCQ imply USC (for inequalities, the implication follows from 
Gordan's theorem). Unlike classical regularity conditions, which require 
differentiability of the active constraints, USC remains applicable when 
the terminal constraints are nondifferentiable, where LICQ and MFCQ are 
not defined; the converse implication does not hold in general. Assuming 
exactness, we derive transversality conditions for all major cases: fixed 
and free terminal time, equality and inequality constraints, moving 
manifolds, and free left endpoint. The approach yields these classical 
results in a concise and transparent manner, avoiding the need for 
constructing cones of endpoint variations or applying separation theorems. 
From the computational perspective, the exact penalty formulation allows 
the differential constraints to be eliminated from the optimization 
entirely, so that only the cost functional and the terminal constraints 
are penalized. A numerical implementation for the time-optimal control of 
a harmonic oscillator confirms the practical applicability of the approach: 
the computed optimal time agrees with the analytical value obtained via 
Pontryagin's maximum principle to high accuracy. This 
establishes USC as a unifying regularity concept encompassing classical 
LICQ and MFCQ and extending naturally to nonsmooth settings.
\end{abstract}
\subsection*{Keywords}
optimal control, 
regularity conditions, 
LICQ, 
MFCQ, 
subdifferentials,
constructive nonsmooth analysis, 
transversality conditions, 
exact penalty methods,
Pontryagin's maximum principle,  
unified separation condition (USC)

\subsection*{MSCcodes}
49K15, 49J52, 49K27, 49J53, 49K30, 49M20

\section{Introduction}

Penalty methods have a long and distinguished history in optimization theory. The first 
rigorous application of penalty functions to variational problems dates back to Courant 
\cite{Courant1943}, who in 1943 used them to study constrained motions. During the 1960s, 
penalty methods became widely adopted for numerical solution of constrained optimization 
problems \cite{Zangwill1967, Eremin1967}. A particularly important development was the 
introduction of exact penalty functions \cite{Eremin1967, Zangwill1967}, which allow 
constrained problems to be reformulated as unconstrained ones without loss of precision. 
For such functions, there exists a finite penalty parameter value such that the solution 
of the original problem coincides with the solution of the penalized problem. The 
foundations of exact penalty theory were laid by Eremin \cite{Eremin1967} and Zangwill 
\cite{Zangwill1967}, and subsequently developed by Demyanov and his school 
\cite{Dem2005, Dem2000, DGP1998, DemRub1995}.

In recent decades, exact penalty methods have found their way into optimal control theory. 
Demyanov, Giannessi, and Karelin \cite{DGK1998} demonstrated the possibility of using 
exact penalties to derive necessary optimality conditions in control problems. Of 
particular importance is the work of Demyanov \cite{Dem2005}, where the exact penalty 
approach was applied to the simplest problem of the calculus of variations. It was shown 
that the Euler equation can be obtained directly from the analysis of the variation of 
the penalty functional. Further developments in this direction include the work of Demyanov and Tamasyan \cite{DT2011} on 
isoperimetric problems and Dolgopolik \cite{Dolgopolik2019} on exact penalisation of 
terminal and state constraints.

In his monograph \cite{Dem2005} (Chapter 6), Demyanov applies exact penalty methods to 
the general optimal control problem. Using needle variations, he derives Pontryagin's 
maximum principle -- the fundamental necessary condition of optimality. The classical 
formulation of this principle can be found in \cite{Pontryagin1983}, and its numerical 
aspects are discussed in \cite{Moiseev1971}. The derivation of transversality conditions 
is presented in these works using the classical technique of cone variations. However, 
these expositions do not address the derivation of transversality conditions within the 
framework of exact penalty methods, nor do they discuss how the regularity assumptions 
can be formulated in a unified way. The present work aims to fill this gap. 

The study of transversality conditions and regularity assumptions has a rich history. 
The classical regularity conditions for equality and inequality constraints were 
systematized by Mangasarian and Fromovitz \cite{Mangasarian1967}. The relationship 
between these conditions and Gordan's theorem \cite{Gordan1873} is well-known in 
optimization theory and plays a crucial role in our analysis. More recent developments in 
nonsmooth analysis \cite{Clarke1983, Burke1999, abbasov2020optimality} provide the 
necessary tools for working with subdifferentials and exact penalty functions. 
Second-order subdifferential calculus and its applications to stability have been 
developed by Mordukhovich and Rockafellar \cite{MordukhovichRockafellar2012}, while 
comprehensive treatments of variational analysis can be found in the monographs 
\cite{Mordukhovich2006I, Mordukhovich2006II}. Regularity properties of collections 
of sets were studied by Kruger \cite{Kruger2006}, and a systematic account of 
metric regularity theory is given by Ioffe \cite{Ioffe_2017}. For optimal control problems 
with state constraints, a maximum principle was established by Vinter and Pappas 
\cite{VinterPappas1982}; more recent contributions include penalty-based approaches 
for mixed constraints \cite{dePinhoFerreiraSmirnov2024}. 
The analysis of second-order 
sufficient conditions and strong metric subregularity in optimal control has been 
the subject of intensive research \cite{Osmolovskii2022, Osmolovskii2023, OsmolovskiiVeliov2023}. 
A comprehensive treatment of second-order variational analysis and its applications 
can be found in the recent monograph by Mordukhovich \cite{Mordukhovich_2024}. 
The convergence of penalty methods for structured problems has been studied by Hu 
and Ralph \cite{HuRalph2004}. 

We consider the exact penalty function \(\phi(x,u)\) constructed as a sum 
of an \(L_2\)-penalty for the differential constraint and a terminal penalty 
defined as the maximum of the absolute values of the equality constraints 
and the positive parts of the inequality constraints. This formulation treats 
equality and inequality constraints in a symmetric way: the absolute value 
handles the two-sided nature of equalities, while the positive part penalizes 
only violations of the inequality constraints. The key regularity 
condition is that the origin is uniformly separated from the subdifferential of the 
penalty function in a neighborhood of the admissible set. For pure terminal constraints, 
this condition is implied by the classical Mangasarian--Fromovitz condition 
for inequalities and by linear independence of gradients for equalities. Unlike classical 
regularity conditions, which require differentiability of the active constraints, 
the proposed condition remains applicable even when the terminal 
constraints are nondifferentiable. Using only classical variations and 
subdifferential calculus, and assuming the exactness of the penalty function (which 
can be established following the technique developed by Demyanov), we derive 
transversality conditions for all major cases: fixed and free terminal time, equality 
and inequality constraints, moving manifolds, and free left endpoint. This approach 
yields these classical results in a concise and transparent manner, avoiding the need 
for constructing cones of endpoint variations or applying separation theorems.

The main contribution of this paper is twofold. From the theoretical 
perspective, we show that the transversality conditions, classically 
derived using cones of endpoint variations and separation theorems, 
can be obtained directly from the variation of an exact penalty 
functional and elementary properties of subdifferentials. This 
eliminates the need for constructing cones and applying separation 
theorems, and extends naturally to nonsmooth terminal constraints, 
where classical constraint qualifications are not defined. From the 
computational perspective, a distinctive feature of the exact penalty 
formulation — and arguably its main practical advantage — is that 
the differential constraints can be eliminated from the optimization 
entirely, leaving only the cost functional and the terminal constraints. 
This is possible because the exact penalty theory guarantees that the 
minimum of the penalized functional is attained at a solution of the 
original problem. Consequently, at each iteration the state trajectory 
can be obtained by numerical integration (e.g., by the Runge--Kutta 
method) for the current control, without including the differential 
constraint in the penalized objective. This property is specific to 
exact penalties: for smooth (e.g., quadratic) penalties, the 
differential constraints cannot be eliminated in this way.

The paper is organized as follows. Section 2 states the optimal control problem. 
Section 3 recalls basic definitions and results from exact penalty theory. 
Section 4 introduces the exactness condition for the penalty function, which we 
propose to call the Unified Separation Condition (USC), and shows that the 
classical regularity conditions appear as its special case. Section 5 derives transversality 
conditions for fixed terminal time. Section 6 generalizes these results to free 
terminal time and moving manifolds. Section 7 presents numerical experiments, 
and Section 8 concludes the paper.

\section{Problem Formulation}
\label{sec:problem}

Consider the classical optimal control problem in Mayer form. We seek a control 
\(u(\cdot) \in U\), where 
\[
U = \left\{ u \in P_m(\mathbb{R}) \mid u(t) \in V \ \forall t \in [t_0, T] \right\},
\]
\(V \subset \mathbb{R}^m\) is a compact set, \(P_m(\mathbb{R})\) is the space of piecewise continuous 
\(m\)-dimensional vector-fun\-ctions, and a corresponding state trajectory 
\(x(\cdot)\) in the space \(P_n^1(\mathbb{R})\) of piecewise continuously 
differentiable \(n\)-dimensional vector-func\-ti\-ons that minimize the func\-ti\-onal
\begin{equation}
    J(x,u) = \Phi^0(x(T), T),
    \label{eq:cost}
\end{equation}

subject to the differential constraint
\begin{equation}
    \dot{x}(t) = f(x(t), u(t), t), \quad t \in [t_0, T],
    \label{eq:state}
\end{equation}
the fixed initial condition
\begin{equation}
    x(t_0) = x_0,
    \label{eq:init}
\end{equation}
and terminal constraints (equalities or inequalities)
\begin{equation}
     \begin{cases}
            \Phi^k(x(T), T)= 0, & k \in E,\\
            \Phi^k(x(T), T)\leq 0, & k \in I,
        \end{cases}
    \label{eq:end_constraints}
\end{equation}
where \(E\) and \(I\) denote the index sets for equality and inequality 
constraints respectively, with \(|E| + |I| = p \le n+1\).

The initial time \(t_0\) is assumed to be fixed, while the terminal time \(T\) 
may be either fixed or free (subject to optimization). The functions 
\(\Phi^0\), \(\Phi^k\), \(k \in E \cup I\), are assumed continuously 
differentiable unless otherwise stated.

We assume that \(f\) is differentiable in \(x\) and Lipschitz in \(x\) uniformly in \(u\) and \(t\), and 
that \(f\) and \(\frac{\partial f}{\partial x}\) are continuous in all arguments. 
At points of discontinuity of \(u\), the derivative \(\dot{x}(t)\) is understood 
in the right-hand sense. Under these assumptions, for every admissible control 
\(u(\cdot)\) there exists a unique piecewise continuously differentiable 
solution \(x(\cdot, u)\) of the initial value problem \eqref{eq:state}--\eqref{eq:init}.

\section{Exact Penalty Functions: Basic Definitions and Results}
\label{sec:penalty}

In this section we recall the main definitions and theorems from the theory of 
exact penalty functions, following \cite{Dem2005, Dem2000}. These results will be 
used to justify the reduction of the constrained problem to an unconstrained one.

Let \([X,\rho]\) be a metric space and \(\Omega \subset X\) a nonempty set. 
Suppose that \(\Omega\) is described in the form
\begin{equation}
    \Omega = \{\xi \in X \mid \phi(\xi) = 0\},
    \label{eq:omega_def}
\end{equation}
where \(\phi: X \to \mathbb{R}\) satisfies \(\phi(\xi) \ge 0\) for all \(\xi \in X\). 
Consider the problem of minimizing a functional \(\mathcal{J}: X \to \mathbb{R}\) over \(\Omega\):
\begin{equation}
    \mathcal{J}(\xi) \to \min_{\xi \in \Omega}.
    \label{eq:original_problem}
\end{equation}

For a fixed \(\lambda \ge 0\), introduce the penalty function
\begin{equation}
    F_\lambda(\xi) = \mathcal{J}(\xi) + \lambda \phi(\xi).
    \label{eq:penalty_function}
\end{equation}

\begin{definition}
    A number \(\lambda^* \ge 0\) is called an \textit{exact penalty constant} for 
    problem \eqref{eq:original_problem} if
    \begin{equation*}
        \inf_{\xi \in X} F_{\lambda^*}(\xi) = \inf_{\xi \in \Omega} \mathcal{J}(\xi).
        \label{eq:exact_penalty_const}
    \end{equation*}
    The function \(F_\lambda\) with \(\lambda > \lambda^*\) is called an \textit{exact 
    penalty function}.
\end{definition}

To state sufficient conditions for the existence of an exact penalty constant, 
we need the notion of the lower derivative.

\begin{definition}
    The lower derivative of \(\phi\) at \(\xi\) is defined as
    \begin{equation*}
        \phi^\downarrow(\xi) = \liminf_{y \to \xi} \frac{\phi(y) - \phi(\xi)}{\rho(\xi,y)}.
        \label{eq:lower_derivative}
    \end{equation*}
\end{definition}

\begin{theorem}[Sufficient conditions for exact penalty]\label{thm:exact_penalty_sufficient}
    Assume that the following conditions hold:
    \begin{enumerate}
        \item \(\inf\limits_{\xi\in X} \mathcal{J}(\xi) > -\infty\);
        \item there exists \(\lambda_0 < \infty\) such that for every \(\lambda \ge \lambda_0\) 
              there is an \(\xi_\lambda \in X\) with \(F_\lambda(\xi_\lambda) = \inf\limits_{\xi\in X} F_\lambda(\xi)\);
        \item there exist \(\delta > 0\) and \(a > 0\) such that
              \begin{equation}
                  \phi^\downarrow(\xi) \leq -a < 0 \quad \forall \xi \in \Omega_\delta \setminus \Omega,
              \end{equation}
              where \(\Omega_\delta = \{\xi \in X \mid \phi(\xi) < \delta\}\);
        \item the function \(\mathcal{J}\) is Lipschitz on \(\Omega_\delta \setminus \Omega\).
    \end{enumerate}
    Then there exists \(\lambda^* \ge \lambda_0\) such that for all \(\lambda > \lambda^*\)
    \begin{equation}
        \phi(\xi_\lambda) = 0, \quad \mathcal{J}(\xi_\lambda) = \inf_{\xi\in\Omega} \mathcal{J}(\xi),
    \end{equation}
    i.e., \(\xi_\lambda\) is a solution of the original problem.
\end{theorem}

\begin{theorem}[Equivalence of local minima] \label{thm:local_minima_equivalence}
    Let the conditions of Theorem \ref{thm:exact_penalty_sufficient} hold. Then there exists 
    \(\lambda^* < \infty\) such that for all \(\lambda > \lambda^*\):
    \begin{enumerate}
        \item[(i)] Any local minimizer of \(F_\lambda\) on \(\Omega_\delta\) is a local 
              minimizer of \(\mathcal{J}\) on \(\Omega\).
        \item[(ii)] Any local minimizer of \(\mathcal{J}\) on \(\Omega\) is a local minimizer of 
              \(F_\lambda\) on \(\Omega_\delta\).
    \end{enumerate}
\end{theorem}

These theorems justify the reduction of the constrained problem to the 
unconstrained minimization of the penalty function for sufficiently large values 
of the penalty parameter. 

\subsection{Penalty function for the optimal control problem}

For our optimal control problem \eqref{eq:cost}--\eqref{eq:end_constraints}, 
where the cost functional $\mathcal{J}$ is \(J(x,u) = \Phi^0(x(T),T)\), we perform the change 
of variables \(z = \dot{x}\). The state trajectory is then recovered as 
\(x(t) = x_0 + \int_{t_0}^t z(\tau)\,d\tau\), and the cost becomes
\begin{equation}
    J(z,u) = \Phi^0\!\left(x_0 + \int_{t_0}^T z(t)\,dt, T\right).
    \label{eq:cost_z}
    \end{equation}
The penalty term \(\phi\) is constructed explicitly as a sum of two components:
\begin{equation}
    \phi(z,u) = \phi_{\mathrm{diff}}(z,u) + \phi_{\mathrm{term}}(z,u),
    \label{eq:phi_decomposition}
\end{equation}
where \(\phi_{\mathrm{diff}}\) is given by 
\begin{align}
    \phi_{\mathrm{diff}}(z,u) = \left[\int_{t_0}^T \left(z(t) - f\left(x_0 + \int_{t_0}^t z(s)\,ds,u(t),t\right)\right)^2 \,dt \right]^{1/2}, \label{eq:phi_diff}
\end{align}
 and the terminal part is
\begin{equation}
    \phi_{\mathrm{term}}(x,u) = \max_{k\in E,\; j\in I}\big\{ |\Phi^k(x(T),T)|, \; [\Phi^j(x(T),T)]_+ \big\},
    \label{eq:phi_term}
\end{equation}
where \([\Phi^j]_+ = \max\{\Phi^j, 0\}\). This unified formulation treats 
equality and inequality constraints in a symmetric way, as discussed in 
the introduction.

Obviously, $$\Omega=\left\{(z,u)\mid \phi(z,u)=0,\; z\in P_n(\mathbb{R}),\: u\in U\right\}=\bigcup_{u\in U}\Omega_u,$$ where
$$\Omega_u=\left\{(z,u)\mid \phi(z,u)=0,\; z\in P_n(\mathbb{R})\right\}.$$ For a fixed control \(u(\cdot)\), the differential equation \eqref{eq:state} with 
initial condition \eqref{eq:init} has a unique solution \(x(\cdot,u)\). Hence, 
for each \(u\) there is a uniquely determined state trajectory, and if 
 \begin{equation}\label{fixed_u_penaly_func}
J(z,u)+\lambda\phi(z,u)
\end{equation}
is exact for any $u\in U_{\mathrm{ad}}$, where
$$U_{\mathrm{ad}} = \left\{ u \in U \mid \Phi^k(x(T,u),T)=0,\ \Phi^j(x(T,u),T)\le 0\ k\in E,\ j\in I \right\},$$
then the only point of $\Omega_u$ is the minimizer of \eqref{fixed_u_penaly_func} and therefore the initial problem is equivalent to the minimization of \eqref{fixed_u_penaly_func} subject to $u\in U_{\mathrm{ad}} $. Moreover, under the regularity conditions introduced in Section~\ref{sec:regularity_condition} (which we will later refer to as the Unified Separation Condition), the same penalty parameter \(\lambda\) can be chosen uniformly for all admissible controls; this can be 
shown by adapting the technique developed in \cite{Dem2005} for similar purposes.

\section{Exactness of the Penalty Function}
\label{sec:exactness}

We now show that under a suitable regularity condition, the penalty function
\(F_\lambda = \mathcal{J} + \lambda\phi\) for the problem \eqref{eq:cost}--\eqref{eq:end_constraints} is exact for sufficiently large 
\(\lambda\). The key condition is that the origin does not belong to the 
subdifferential of \(\phi\) in a neighborhood of the admissible set $\Omega$.

We assume that there exists a neighborhood \(B\) of the admissible set $\Omega$ and a 
  positive constant \(a \) such that for all \((z,u) \in B\) with \(\phi(z,u) > 0\),
 \begin{equation}\label{ass:regularity}
    \operatorname{dist}\bigl(0, \partial\phi(z,u)\bigr) \ge a > 0.
\end{equation}
Throughout the paper, \(\partial\) denotes the subdifferential in the 
sense of convex analysis. Since \(\phi_{\mathrm{diff}}\) is a norm in 
\(L_2\) and \(\phi_{\mathrm{term}}\) is a maximum of convex functions \(|\Phi^k|\) and \([\Phi^j]_+\), the penalty function 
\(\phi = \phi_{\mathrm{diff}} + \phi_{\mathrm{term}}\) is convex. 
Consequently, the convex subdifferential coincides with the Clarke 
subdifferential and the limiting subdifferential of Mordukhovich, and 
the framework of constructive nonsmooth analysis 
\cite{Dem2005, DemRub1995} applies without ambiguity.

We propose to call condition \eqref{ass:regularity} the 
{Unified Separation Condition (USC)}.

Under this assumption, for any \((z,u) \in B\), the 
subdifferential \(\partial\phi(z,u)\) is a convex compact set not containing 
the origin. By the separation theorem in Hilbert space, there exists a 
direction \(h\) (depending on \((z,u)\)) such that
\[
\phi'(z,u; h) = \max_{v \in \partial\phi(z,u)} \langle v, h \rangle \le -a < 0,
\]
hence \(\phi^\downarrow(z,u) \le -a < 0\). Theorem~\ref{thm:exact_penalty_sufficient} 
then guarantees the existence of a finite \(\lambda^*\) such that for all 
\(\lambda > \lambda^*\) the penalty function \(F_\lambda\) is exact.

Condition \ref{ass:regularity} can be analyzed in terms of the components 
of \(\phi = \phi_{\mathrm{diff}} + \phi_{\mathrm{term}}\). Recall (see \cite{Dem2005} for details) that  when \(\phi_{\mathrm{diff}}(z,u) = 0\) the 
subdifferential of \(\phi_{\mathrm{diff}}\) is given by formula:
\[
\partial\phi_{\mathrm{diff}}(z,u) = \left\{ v(\cdot) \;\middle|\; \|v(t)\| \le 1,\; 
v(t) = \int_t^T \left(\frac{\partial f(x(\tau),u(\tau),\tau)}{\partial x}\right)^* v(\tau)\,d\tau \right\},
\]
and when \(\phi_{\mathrm{diff}}(z,u) > 0\), this set reduces to a singleton 
\(\{\nabla\phi_{\mathrm{diff}}(z,u)\}\) with
\[
\nabla\phi_{\mathrm{diff}}(z,u)(t) = w(t) - \int_t^T \left(\frac{\partial f}{\partial x}\right)^* w(\tau)\,d\tau,
\]
\[
w(t) = \frac{1}{\phi_{\mathrm{diff}}(z,u)}\left( z(t) - f\!\left( x_0 + \int_{t_0}^t z(\tau)\,d\tau,\; u(t), t \right) \right),
\]
and \(\|w\|_{L_2} = 1\).

For the terminal part, let \(x_T = x_0 + \int_{t_0}^T z(t)\,dt\),
and define the sets of active constraints:
\[
E_{\mathrm{act}}(z) = \left\{ k \in E \;\middle|\; |\Phi^k(x_T,T)| = \phi_{\mathrm{term}}(z) \right\},
\]
\[
I_{\mathrm{act}}(z) = \left\{ j \in I \;\middle|\; \Phi^j(x_T,T) = \phi_{\mathrm{term}}(z) \right\}.
\]
Then it is obvious that for \(\phi_{\mathrm{term}}(z) > 0\) we have
\begin{multline*}
\partial\phi_{\mathrm{term}}(z) = \operatorname{co}\biggl\{ \nabla\Phi^k(x_T,T)\,\operatorname{sign}\Phi^k(x_T,T),\; \nabla\Phi^j(x_T,T) \\
\;\bigg|\; k\in E_{\mathrm{act}}(z),\; j\in I_{\mathrm{act}}(z) \biggr\},
\end{multline*}
and when \(\phi_{\mathrm{term}}(z) = 0\), the 
subdifferential is
\begin{multline*}
\partial\phi_{\mathrm{term}}(z) = \operatorname{co}\biggl\{ \nabla\Phi^k(x_T,T),\; -\nabla\Phi^k(x_T,T),\; \nabla\Phi^j(x_T,T) \\
\;\bigg|\; k\in E_{\mathrm{act}}(z),\; j\in I_{\mathrm{act}}(z)\biggr\}.
\end{multline*}


\subsection{The Unified Separation Condition (USC)}
\label{sec:regularity_condition}

We now analyze condition \eqref{ass:regularity} in terms of the 
components \(\phi = \phi_{\mathrm{diff}} + \phi_{\mathrm{term}}\). 
Recall that the subdifferential of \(\phi_{\mathrm{diff}}\) is given 
by Demyanov's formula, and the subdifferential of \(\phi_{\mathrm{term}}\) 
depends on the active terminal constraints.

Without loss of generality, in what follows we assume that all constraints 
in \(E\cup I\) are active at the optimum, i.e., \(\Phi^k(x^*(T),T)=0\) for all 
\(k\in E\cup I\). Inactive constraints can be ignored as they do not affect 
the local analysis.

We now examine what this condition implies in three cases: when only terminal 
constraints are violated, when only the differential constraint is violated, 
and when both are violated.

\paragraph{Pure terminal constraints (\(\phi_{\mathrm{diff}} \equiv 0,\ \phi_{\mathrm{term}} > 0\))}
In this case, condition \eqref{ass:regularity} reduces to the uniform 
separation of \(\partial\phi_{\mathrm{term}}(z)\) from the origin in 
a neighborhood of the admissible set. This holds when the classical 
regularity conditions (LICQ for equalities, MFCQ for inequalities) 
are satisfied (see below).

\paragraph{Pure differential constraint (\(\phi_{\mathrm{diff}} > 0,\ \phi_{\mathrm{term}} \equiv 0\))}
Here \(\partial\phi = \{\nabla\phi_{\mathrm{diff}}\}\), and condition 
\eqref{ass:regularity} reduces to \(\|\nabla\phi_{\mathrm{diff}}\| \ge a > 0\) 
uniformly in a neighborhood of the admissible set, where \(\nabla\phi_{\mathrm{diff}}\) 
is given by Demyanov's formula.

\paragraph{Both constraints violated (\(\phi_{\mathrm{diff}} > 0,\ \phi_{\mathrm{term}} > 0\))}
Here \(\partial\phi = \{\nabla\phi_{\mathrm{diff}}\} + \partial\phi_{\mathrm{term}}\), 
and condition \eqref{ass:regularity} reduces to the uniform separation 
of \(-\nabla\phi_{\mathrm{diff}}\) from \(\partial\phi_{\mathrm{term}}\) 
in a neighborhood of the admissible set.

The classical regularity conditions (LICQ, MFCQ) imply condition 
\eqref{ass:regularity} when \(\phi_{\mathrm{diff}} = 0\) and the 
terminal constraints are active. A detailed analysis is given in the 
following subsections.

\subsubsection{Only Equality Constraints}

Suppose $\phi_{\mathrm{diff}} \equiv 0$ and only equality constraints are present.
Then $\phi = \phi_{\mathrm{term}}$.

\begin{theorem}[LICQ $\Rightarrow$ USC for equalities]
\label{thm:licq_usc}
    If the gradients of the equality constraints, i.e. $\nabla\Phi^k(x^*(T),T)$, $k\in E$, are linearly independent,
    then there exist a neighborhood $B$ of $(z^*,u^*)$ and $a>0$ such that
    $\operatorname{dist}(0,\partial\phi_{\mathrm{term}}(z))\ge a$ for all $(z,u)\in B$
    with $\phi_{\mathrm{term}}(z)>0$. Hence USC holds.
\end{theorem}

\begin{proof}
Take a point with $\phi_{\mathrm{term}}(z)>0$ near the optimum. By continuity,
all constraints are active. The subdifferential is
$\partial\phi_{\mathrm{term}}(z)=\operatorname{co}\{\nabla\Phi^k\operatorname{sign}\Phi^k\}$.
Linear independence of the gradients at the optimum survives in a neighborhood,
and signs are locally constant. A convex hull of linearly independent vectors
does not contain the origin; the distance to the origin is positive and
remains uniformly bounded away from zero by a standard compactness argument.
Hence USC holds. \end{proof}

Thus LICQ $\Rightarrow$ USC, and USC extends beyond the smooth case (see example below).

\subsubsection{Only Inequality Constraints}

Assume $\phi_{\mathrm{diff}}\equiv0$, only inequalities. Then $\phi=\phi_{\mathrm{term}}$.

At the optimum, $\partial\phi_{\mathrm{term}}(z^*)=\operatorname{co}\{\nabla\Phi^j\mid j\in I_{\mathrm{act}}\}$.
MFCQ is equivalent to $0\notin\partial\phi_{\mathrm{term}}(z^*)$ by Gordan's theorem \cite{Gordan1873},
so $d:=\operatorname{dist}(0,\partial\phi_{\mathrm{term}}(z^*))>0$.

\begin{theorem}[USC for nondifferentiable constraints]
\label{thm:usc_nonsmooth}
Let \(\Phi^k\), \(k \in E \cup I\), be locally Lipschitz, and suppose 
that at an admissible point \(z^*\)
\[
0 \notin \partial \phi_{\mathrm{term}}(z^*), \qquad 
d := \operatorname{dist}(0, \partial \phi_{\mathrm{term}}(z^*)) > 0.
\]
Then there exist a neighborhood \(B\) of \(z^*\) and a constant \(a > 0\) 
such that
\[
\operatorname{dist}(0, \partial \phi_{\mathrm{term}}(z)) \ge a 
\quad \forall z \in B \text{ with } \phi_{\mathrm{term}}(z) > 0.
\]
Hence USC holds in \(B\).
\end{theorem}

\begin{proof}
Since \(\phi_{\mathrm{term}}\) is locally Lipschitz (as a maximum of 
locally Lipschitz functions), its subdifferential 
\(\partial \phi_{\mathrm{term}}\) is upper semicontinuous. For 
\(\varepsilon = d/2\), there exists a neighborhood \(B\) of \(z^*\) such 
that for any \(v \in \partial \phi_{\mathrm{term}}(z)\), there exists 
\(v_0 \in \partial \phi_{\mathrm{term}}(z^*)\) with 
\(\|v - v_0\| \le \varepsilon\). Then 
\(\|v\| \ge \|v_0\| - \varepsilon \ge d - d/2 = d/2\). Hence 
\(\operatorname{dist}(0, \partial \phi_{\mathrm{term}}(z)) \ge d/2\) in \(B\).
\end{proof}

\begin{example}[Nondifferentiable terminal constraint]
\label{ex:nonsmooth}
Consider the one-dimensional terminal constraint
\[
\Phi(x(T)) = |x(T)| - 2 = 0.
\]
The function \(\Phi\) is not differentiable at the admissible points 
\(x(T) = \pm 2\), so classical LICQ is not defined. The terminal penalty 
is \(\phi_{\mathrm{term}}(x) = \bigl||x(T)| - 2\bigr|\). For 
\(x(T) \ne 2\) in a neighborhood of \(x(T) = 2\), we have 
\(\phi_{\mathrm{term}}(x) = |x(T) - 2|\) and
\[
\partial \phi_{\mathrm{term}}(x) = 
\begin{cases} 
\{-1\}, & x(T) < 2, \\ 
\{1\}, & x(T) > 2. 
\end{cases}
\]
Hence \(\operatorname{dist}(0, \partial \phi_{\mathrm{term}}(x)) = 1\), 
and by Theorem~\ref{thm:usc_nonsmooth}, USC holds with \(a = 1\). 
The same holds near \(x(T) = -2\).

This example illustrates how the exact penalty framework naturally 
extends to problems where classical regularity conditions are not 
defined.
\end{example}

\subsubsection{Mixed Constraints}

Assume $\phi_{\mathrm{diff}}\equiv0$ and both equalities and inequalities are present.
Classical regularity: $\{\nabla\Phi^k\}_{k\in E}$ are linearly independent, and
there exists $d$ with $\langle\nabla\Phi^k,d\rangle=0$ ($k\in E$),
$\langle\nabla\Phi^j,d\rangle<0$ ($j\in I$).

For points with $\phi_{\mathrm{term}}(z)>0$ close to the optimum, active constraints
satisfy the same properties. Hence $\partial\phi_{\mathrm{term}}(z)$ is the convex
hull of a set of vectors that does not contain zero, and a uniform positive
lower bound on the distance follows as in Theorem~\ref{thm:licq_usc}. Thus USC holds.

The conditions above are stated for fixed $T$; the case of free $T$ 
is treated in Sections~\ref{sec:free_time} and~\ref{sec:moving_manifold}.

\section{Transversality Conditions for Fixed Time}
\label{sec:transversality_fixed}

First, let us recall that for problem \eqref{eq:cost}--\eqref{eq:end_constraints}, Pontryagin's 
maximum principle (see \cite{Pontryagin1983, Dem2005}) states that if 
\((x^*,u^*)\) is an optimal pair, then there exists an adjoint variable 
\(\psi(t)\) satisfying
\begin{equation}
    \dot{\psi}(t) = - \left(\frac{\partial f}{\partial x}(x^*,u^*,t)\right)^T \psi(t),
    \label{eq:adjoint_general}
\end{equation}
such that for almost every \(t \in [t_0, T]\) the Hamiltonian
\(H = \langle \psi, f \rangle\) attains its maximum:
\begin{equation}
    H(x^*(t), u^*(t), \psi(t), t) = \max_{v \in V} H(x^*(t), v, \psi(t), t).
    \label{eq:maximum_principle}
\end{equation}

In this section we prove that the transversality conditions can be 
obtained directly from the analysis of the variation of the exact 
penalty functional. Throughout, we assume that the terminal time \(T\) 
is fixed and that USC \eqref{ass:regularity} holds.

\subsection{Variation of the Penalty Functional}
\label{sec:variation}

We now compute the variation of each term in \(F_\lambda\) defined by 
\eqref{eq:penalty_function} with \(\phi\) given by \eqref{eq:phi_decomposition}. 
Recall that \(\phi_{\mathrm{diff}}(x,u) = \| \dot{x} - f(x,u,t) \|_{L_2}\) and 
$$\phi_{\mathrm{term}}(x,u) = \max_{k\in E,\; j\in I}\big\{ |\Phi^k(x(T),T)|, \; [\Phi^j(x(T),T)]_+\big\}.$$
\noindent\textit{Terminal cost:}
\begin{equation*}
    \delta \Phi^0 = \left\langle \frac{\partial \Phi^0(x^*(T))}{\partial x}, \xi(T) \right\rangle.
    \label{eq:var_terminal}
\end{equation*}

\noindent\textit{Penalty for differential constraint:}
At the optimal point, \(\dot{x}^* - f(x^*,u^*,t) = 0\), so \(\phi_{\mathrm{diff}} = 0\). 
The subdifferential of the norm at zero is the unit ball: 
\(\partial \|\cdot\|(0) = \{ s \mid \|s\|_{L_2} \le 1 \}\). 
Therefore, the first variation of \(\phi_{\mathrm{diff}}\) is given by
\begin{equation*}
    \delta \phi_{\mathrm{diff}} = \max_{\|s(\cdot)\|_{L_2} \le 1} \int_{t_0}^T \langle s(t), \dot{\xi}(t) - f_x^*(t) \xi(t) - f_u^*(t) v(t) \rangle \,dt,
    \label{eq:var_penalty_state}
\end{equation*}
where \(f^*(t) = f(x^*(t),u^*(t),t)\). This follows from the fact that for any 
convex Lipschitz function \(\varphi\) with \(\varphi(\zeta^*) = 0\), the first variation is 
\(\delta \varphi = \displaystyle\max_{s \in \partial \varphi(\zeta^*)} \langle s, \delta \zeta \rangle\).

\noindent\textit{Penalty for terminal constraints:}
At the optimal point, for equality constraints \(\Phi^k = 0\) and for active 
inequality constraints also \(\Phi^k = 0\). The subdifferential of the maximum 
function at a point where several terms attain the maximum is the convex hull 
of the subdifferentials of the active terms. Using the subdifferentials \(\partial |\cdot|(0) = [-1,1]\) and 
\(\partial [\cdot]_+(0) = [0,1]\), together with the chain rule, we obtain
\[
\partial |\Phi^k|(x) = [-1,1] \cdot \nabla\Phi^k(x), \qquad
\partial [\Phi^k]_+(x) = [0,1] \cdot \nabla\Phi^k(x) \quad \text{when } \Phi^k(x)=0,
\]
and therefore
\begin{align}
    \delta |\Phi^k| &= \max_{\alpha_k \in [-1,1]} \alpha_k \left\langle \frac{\partial \Phi^k}{\partial x}, \xi(T) \right\rangle, \quad k \in E,\\
    \delta [\Phi^k]_+ &= \max_{\beta_k \in [0,1]} \beta_k \left\langle \frac{\partial \Phi^k}{\partial x}, \xi(T) \right\rangle, \quad k \in I.
\end{align}
Since \(\phi_{\mathrm{term}}\) is the maximum over all these terms, its first 
variation is the maximum over the convex hull of the variations of the active 
terms. Consequently,
\begin{equation*}
    \delta \phi_{\mathrm{term}} = \max_{\substack{\alpha_k \in [-1,1],\; k\in E \\ \beta_k \in [0,1],\; k\in I}} \left\langle \sum_{k\in E} \alpha_k \frac{\partial \Phi^k}{\partial x} + \sum_{k\in I} \beta_k \frac{\partial \Phi^k}{\partial x}, \xi(T) \right\rangle.
    \label{eq:var_phi_term}
\end{equation*}

\noindent\textit{Total variation:}
Since \(F_\lambda = \Phi^0 + \lambda \phi_{\mathrm{diff}} + \lambda \phi_{\mathrm{term}}\), we obtain
\begin{align}
    \delta F_\lambda &= \left\langle \frac{\partial \Phi^0}{\partial x}\bigg|_{t=T}, \xi(T) \right\rangle \nonumber \\
    &\quad + \lambda \max_{\|s(\cdot)\|_{L_2} \le 1} \int_{t_0}^T \langle s(t), \dot{\xi}(t) - f_x^*(t) \xi(t) - f_u^*(t) v(t) \rangle \,dt \nonumber \\
    &\quad + \lambda \max_{\substack{\alpha_k \in [-1,1],\; k\in E \\ \beta_k \in [0,1],\; k\in I}} \left\langle \sum_{k\in E} \alpha_k \frac{\partial \Phi^k}{\partial x} + \sum_{k\in I} \beta_k \frac{\partial \Phi^k}{\partial x}, \xi(T) \right\rangle.
    \label{eq:var_total}
\end{align}

\subsection{Necessary Conditions}
\label{sec:necessary_conditions}
Since the penalty function \(F_\lambda\) is exact for sufficiently large \(\lambda\), 
the original constrained problem is equivalent to the unconstrained 
minimization of \(F_\lambda\) over all admissible \((x,u)\). In this unconstrained 
formulation, the variations \(\xi(\cdot)\) and \(v(\cdot)\) can be chosen 
arbitrarily and independently (subject only to \(\xi(t_0)=0\)). 

From the expression \eqref{eq:var_total} for the total variation, we observe that 
\(\delta F_\lambda\) can be written as a single maximum over all variables 
\(s(\cdot) \in L_2[t_0,T]\) with \(\|s\|_{L_2} \le 1\), \(\alpha_k \in [-1,1]\), 
\(k\in E\), and \(\beta_k \in [0,1]\), \(k\in I\):
\begin{multline}
    \delta F_\lambda = \max_{\substack{\|s\|_{L_2}\le 1 \\ \alpha_k\in[-1,1] \\ \beta_k\in[0,1]}} \Bigg[ \int_{t_0}^T \langle s(t), \dot{\xi}(t) - f_x^*(t) \xi(t) - f_u^*(t) v(t) \rangle \,dt \\
    + \left\langle \frac{\partial \Phi^0}{\partial x}\bigg|_{t=T} + \lambda \bigg( \sum_{k\in E} \alpha_k \frac{\partial \Phi^k}{\partial x}\bigg|_{t=T} + \sum_{k\in I} \beta_k \frac{\partial \Phi^k}{\partial x}\bigg|_{t=T} \bigg), \xi(T) \right\rangle \Bigg].
    \label{eq:delta_F_as_max}
\end{multline}

Integrating by parts the term containing \(\dot{\xi}\) and using 
\(\xi(t_0)=0\) gives
\begin{equation*}
    \int_{t_0}^T \langle s(t), \dot{\xi}(t) \rangle \,dt = \langle s(T), \xi(T) \rangle - \int_{t_0}^T \langle \dot{s}(t), \xi(t) \rangle \,dt.
    \label{eq:integration_by_parts}
\end{equation*}

Substituting this into \eqref{eq:delta_F_as_max} and rearranging yields
\begin{multline}
    \delta F_\lambda = \max_{\substack{\|s\|_{L_2}\le 1 \\ \alpha_k\in[-1,1] \\ \beta_k\in[0,1]}} \Bigg[ -\int_{t_0}^T \langle \dot{s}(t) + (f_x^*(t))^T s(t), \xi(t) \rangle \,dt \\
    + \Bigg\langle s(T) + \frac{\partial \Phi^0}{\partial x}\bigg|_{t=T} + \lambda \bigg( \sum_{k\in E} \alpha_k \frac{\partial \Phi^k}{\partial x}\bigg|_{t=T} + \sum_{k\in I} \beta_k \frac{\partial \Phi^k}{\partial x}\bigg|_{t=T} \bigg), \xi(T) \Bigg\rangle \\
    - \int_{t_0}^T \langle s(t), f_u^*(t) v(t) \rangle \,dt \Bigg].
    \label{eq:delta_F_after_parts}
\end{multline}

The necessary condition \(\delta F_\lambda \ge 0\) must hold for all independent 
variations \(\xi(\cdot)\) and \(v(\cdot)\). To analyze this condition, we consider 
two families of variations separately.

\paragraph{Variations with \(\xi(T) = 0\)}
For variations with \(\xi(T)=0\), the terms involving \(\alpha_k\) and \(\beta_k\) vanish, and the expression simplifies to
\begin{equation*}
    \delta F_\lambda = \max_{\|s\|_{L_2}\le 1} \left[ -\int_{t_0}^T \langle \dot{s}(t) + (f_x^*(t))^T s(t), \xi(t) \rangle \,dt - \int_{t_0}^T \langle s(t), f_u^*(t) v(t) \rangle \,dt \right].
\end{equation*}

Since \(\xi(t)\) and \(v(t)\) can be chosen arbitrarily 
independently, the necessary condition \(\delta F_\lambda \ge 0\) for all such 
variations implies that the maximizing \(s^*\) must satisfy, for almost every 
\(t \in [t_0, T]\),
\begin{equation}
    \dot{s}^*(t) + (f_x^*(t))^T s^*(t) = 0,
    \label{eq:adjoint_s}
\end{equation}
and
\begin{equation}
    (f_u^*(t))^T s^*(t) = 0.
    \label{eq:optimality_u}
\end{equation}
Equation \eqref{eq:adjoint_s} is the adjoint equation; it guarantees that  
\(\dot{s}^*\) is bounded (in fact, continuous on each interval of 
continuity of \(f_x^*\)). Equation 
\eqref{eq:optimality_u} is the optimality condition for the control. For 
problems with control constraints, a more careful analysis using needle 
variations (see \cite{Dem2005}) yields the maximum principle 
\eqref{eq:maximum_principle}.

\paragraph{Variations isolating the boundary term}
To isolate the boundary term, consider a family of variations of the form
\(\xi(t) = \xi_0(t) + \varepsilon\eta(t)\), where \(\xi_0\) is an arbitrary 
variation from the previous class (satisfying \(\xi_0(T)=0\)) and 
 $\eta(t)\in P^1[t_0,T]$ has the form
$$
\eta(t)=
\begin{cases}
0,& \text{ if } t\in [t_0,T-\varepsilon],\\
1+(t-T)/\varepsilon,& \text{ if } t\in [T-\varepsilon,T]
\end{cases}
$$

For such variations, we have \(\xi(T) = \varepsilon\). Substituting into 
\eqref{eq:delta_F_after_parts} and using the fact that the contribution 
from \(\xi_0\) vanishes due to the adjoint equation \eqref{eq:adjoint_s} 
(as established from the previous class of variations), we are left with 
the contribution from the perturbation \(\varepsilon\eta(t)\). Since 
\(\eta(t)\) is supported on an interval of length \(\varepsilon\) and is 
bounded, the integral term involving \(\xi(t)\) is of order \(O(\varepsilon^2)\). 
Hence, for the optimal triple \((s^*,\alpha^*,\beta^*)\)
\begin{multline*}
    \delta F_\lambda = \varepsilon \Bigg\langle s^*(T) + \frac{\partial \Phi^0}{\partial x}\bigg|_{t=T}
    + \lambda \bigg( \sum_{k\in E} \alpha_k^* \frac{\partial \Phi^k}{\partial x}\bigg|_{t=T} + \sum_{k\in I} \beta_k^* \frac{\partial \Phi^k}{\partial x}\bigg|_{t=T} \bigg), \eta(T) \Bigg\rangle + o(\varepsilon).
\end{multline*}
Since \(\eta(T)=1\), the necessary condition \(\delta F_\lambda \ge 0\) for both 
\(\varepsilon > 0\) and \(\varepsilon < 0\) forces the coefficient to vanish. Hence,
\begin{equation}
    s^*(T) + \frac{\partial \Phi^0}{\partial x}\bigg|_{t=T} + \lambda \left( \sum_{k\in E} \alpha_k^* \frac{\partial \Phi^k}{\partial x}\bigg|_{t=T} + \sum_{k\in I} \beta_k^* \frac{\partial \Phi^k}{\partial x}\bigg|_{t=T} \right) = 0.
    \label{eq:transversality_raw}
\end{equation}


Thus, we obtain the adjoint equation \eqref{eq:adjoint_s}, the control 
optimality condition \eqref{eq:optimality_u} (or the maximum principle), 
and the transversality condition \eqref{eq:transversality_raw}. The 
parameters \(\alpha_k^*\) and \(\beta_k^*\) will be identified as the Lagrange 
multipliers associated with the terminal constraints in the classical 
formulation.

\subsection{Main Result: Transversality Theorem}

We now collect the results of Sections~\ref{sec:variation}--\ref{sec:necessary_conditions} 
and formulate the main result of this section.

\begin{theorem}[Transversality, fixed time]
\label{thm:transversality_fixed}
Let \((x^*, u^*)\) be an optimal pair for problem 
\eqref{eq:cost}--\eqref{eq:end_constraints} with fixed \(T\), and assume 
that USC \eqref{ass:regularity} holds. Then there exist multipliers 
\(\nu_k \in \mathbb{R}\) \((k \in E)\) and \(\mu_k \ge 0\) \((k \in I)\) 
such that the adjoint equation
\begin{equation}
    \dot{\psi}(t) = - (f_x(x^*(t), u^*(t), t))^T \psi(t),
    \label{eq:adjoint_thm}
\end{equation}
the maximum principle \eqref{eq:maximum_principle}, and the 
transversality condition
\begin{equation}
    \psi(T) = - \frac{\partial \Phi^0}{\partial x}\bigg|_{t=T} 
    - \sum_{k\in E} \nu_k \frac{\partial \Phi^k}{\partial x}\bigg|_{t=T} 
    - \sum_{k\in I} \mu_k \frac{\partial \Phi^k}{\partial x}\bigg|_{t=T}
    \label{eq:transversality_thm}
\end{equation}
hold.
\end{theorem}

\begin{proof}
By USC and Theorem~\ref{thm:exact_penalty_sufficient}, there exists 
\(\lambda^* < \infty\) such that for all \(\lambda > \lambda^*\) the penalty 
function \(F_\lambda\) is exact. Hence \((x^*, u^*)\) is a local minimizer 
of \(F_\lambda\) on \(B\), and for any admissible variation we have 
\(\delta F_\lambda \ge 0\).

The first variation of \(F_\lambda\) is computed in 
Section~\ref{sec:variation} and is given by \eqref{eq:var_total}. 
Integrating by parts and using \(\xi(t_0) = 0\) yields 
\eqref{eq:delta_F_after_parts}.

For variations with \(\xi(T) = 0\), the boundary terms vanish, and the 
necessary condition \(\delta F_\lambda \ge 0\) implies the adjoint 
equation \eqref{eq:adjoint_s} and the optimality condition 
\eqref{eq:optimality_u}.

For variations isolating the boundary term, using 
\(\xi(t) = \xi_0(t) + \varepsilon\eta(t)\) with \(\eta(T) = 1\), we obtain 
\eqref{eq:transversality_raw}. Introducing \(\psi(t) = \lambda s^*(t)\), 
\(\nu_k = \lambda \alpha_k^*\), and \(\mu_k = \lambda \beta_k^*\), this 
yields \eqref{eq:transversality_thm}.
\end{proof}

%
%

\begin{remark}
The multipliers \(\nu_k, \mu_k\) are bounded by \(\lambda\) in absolute 
value, a consequence of the exact penalty formulation. This bound is 
not restrictive for sufficiently large \(\lambda\) and can be identified 
with the classical Lagrange multipliers.
\end{remark}

\begin{remark}
If the right endpoint is free, condition \eqref{eq:transversality_thm} 
simplifies to
\begin{equation*}
    \psi(T) = - \frac{\partial \Phi^0}{\partial x}\bigg|_{t=T},
    \label{eq:transversality_free}
\end{equation*}
which corresponds to the condition for the Mayer problem with free endpoint.
\end{remark}

\section{Generalizations: Free Terminal Time and Moving Endpoints}
\label{sec:generalizations}

In this section we extend the results obtained above to the cases of free 
terminal time \(T\) and moving right endpoint (time-dependent manifold). 
In all cases, the key regularity condition remains USC \eqref{ass:regularity}.

\subsection{Free Terminal Time (Time-Independent Constraints)}
\label{sec:free_time}

\begin{theorem}[Transversality, free terminal time]
\label{thm:transversality_free_T}
Let \((x^*, u^*, T^*)\) be an optimal triple for problem 
\eqref{eq:cost}--\eqref{eq:end_constraints} with free \(T\) and 
time-independent terminal constraints, and assume USC holds. Then, in 
addition to \eqref{eq:adjoint_thm}, \eqref{eq:maximum_principle}, and 
\eqref{eq:transversality_thm}, the Hamiltonian vanishes at the terminal 
time:
\begin{equation}
    H(x^*(T^*), u^*(T^*), \psi(T^*), T^*) = 0.
    \label{eq:hamiltonian_zero}
\end{equation}
\end{theorem}

\begin{proof}
For an optimal triple \((x^*,u^*,T^*)\) and sufficiently large 
\(\lambda\), exactness holds by Theorem~\ref{thm:exact_penalty_sufficient} 
(applied to the extended variable \((x, u, T)\)). Consider 
variations
\[
x_\varepsilon = x^* + \varepsilon\xi, \quad 
u_\varepsilon = u^* + \varepsilon v, \quad 
T_\varepsilon = T^* + \varepsilon \delta T,
\]
with \(\xi(t_0) = 0\). The endpoint variation is
\[
\delta x(T) = \xi(T) + \dot{x}^*(T) \delta T.
\]
Repeating the analysis of Theorem~\ref{thm:transversality_fixed} and 
including the variation \(\delta T\) yields the adjoint equation, the 
maximum principle, and the transversality condition 
\eqref{eq:transversality_thm}. Because \(T\) is free, the first-order 
optimality condition with respect to \(\delta T\) gives 
\eqref{eq:hamiltonian_zero}.
\end{proof}

\subsection{Moving Manifold (Time-Dependent Terminal Constraints)}
\label{sec:moving_manifold}

Consider now the case where the terminal constraints depend explicitly 
on time:
\begin{equation}
    \Phi^k(x(T), T) = 0, \quad k \in E, \qquad 
    \Phi^k(x(T), T) \le 0, \quad k \in I.
    \label{eq:constraints_time_dep}
\end{equation}

\begin{theorem}[Transversality, moving manifold]
\label{thm:transversality_moving}
Let \((x^*, u^*, T^*)\) be an optimal triple for problem 
\eqref{eq:cost}--\eqref{eq:constraints_time_dep} with free \(T\) and 
time-dependent terminal constraints, and assume USC holds. Then, in 
addition to \eqref{eq:adjoint_thm} and \eqref{eq:maximum_principle}, the 
following transversality conditions hold:
\begin{align}
    \psi(T^*) &= - \frac{\partial \Phi^0}{\partial x}\bigg|_{T^*} 
    - \sum_{k\in E} \tilde{\nu}_k \frac{\partial \Phi^k}{\partial x}\bigg|_{T^*} 
    - \sum_{k\in I} \tilde{\mu}_k \frac{\partial \Phi^k}{\partial x}\bigg|_{T^*}, 
    \label{eq:transv_moving_x}\\
    H(T^*) &+ \frac{\partial \Phi^0}{\partial t}\bigg|_{T^*} 
    + \sum_{k\in E} \tilde{\nu}_k \frac{\partial \Phi^k}{\partial t}\bigg|_{T^*} 
    + \sum_{k\in I} \tilde{\mu}_k \frac{\partial \Phi^k}{\partial t}\bigg|_{T^*} = 0, 
    \label{eq:transv_moving_t}
\end{align}
where \(\tilde{\nu}_k \in \mathbb{R}\) and \(\tilde{\mu}_k \ge 0\).
\end{theorem}

\begin{proof}
Repeating the analysis of Theorem~\ref{thm:transversality_free_T} but 
now taking into account the explicit time dependence of the constraint 
functions, the endpoint variation 
\(\delta x(T) = \xi(T) + \dot{x}^*(T)\delta T\) yields additional terms 
involving \(\partial \Phi^k/\partial t\). Collecting the first-order 
terms in \(\delta T\) gives \eqref{eq:transv_moving_t}.
\end{proof}

\subsection{Free Left Endpoint}
\label{sec:free_left}

Suppose the initial condition is subject to
\begin{equation}
    \chi^l(x(t_0), t_0) = 0 \ (l\in E_L), \qquad 
    \chi^l(x(t_0), t_0) \le 0 \ (l\in I_L).
    \label{eq:left_constraints}
\end{equation}

\begin{theorem}[Transversality, free left endpoint]
\label{thm:transversality_left}
Let \((x^*, u^*)\) be an optimal pair for problem 
\eqref{eq:cost}--\eqref{eq:end_constraints} with free left endpoint 
subject to \eqref{eq:left_constraints}, and assume USC holds. Then
\begin{equation}
    \psi(t_0) = \sum_{l\in E_L} \gamma_l \frac{\partial \chi^l}{\partial x}\bigg|_{t=t_0} 
    + \sum_{l\in I_L} \delta_l \frac{\partial \chi^l}{\partial x}\bigg|_{t=t_0},
    \label{eq:transv_left}
\end{equation}
with \(\gamma_l \in \mathbb{R}\) and \(\delta_l \ge 0\).
\end{theorem}

\begin{proof}
Repeating the arguments of Theorem~\ref{thm:transversality_fixed} and 
Theorem~\ref{thm:transversality_free_T} with \(T\) replaced by \(t_0\), 
the integration direction reversed, and noting 
\(\xi(t_0) \ne 0\), \(\xi(T) = 0\), we obtain \eqref{eq:transv_left}. 
If \(t_0\) is also free, additional conditions involving the Hamiltonian 
and time derivatives of \(\chi^l\) appear, with the sign changed due to 
the direction of integration.
\end{proof}

\section{Numerical Experiments: Time-Optimal Control of a Har\-mo\-nic Oscillator}
\label{sec:numerical}

In this section we demonstrate the theoretical results obtained above on a 
classical time-optimal control problem. The purpose of these experiments is 
twofold: first, to illustrate the applicability of the exact penalty approach; 
second, to provide an independent verification of the theoretical derivations 
by comparing the numerical results with the known analytical solution. 
A distinctive feature of the exact penalty formulation is that the 
differential constraints can be eliminated from the optimization entirely, 
so that only the cost functional and the terminal constraints are penalized.

We consider a harmonic oscillator with the objective to transfer the system 
from the initial state to an arbitrary state with zero velocity in minimum time.

\subsection{Problem Statement}

We consider the problem of minimizing time:
\begin{equation*}
\begin{split}
    &T \longrightarrow \min\\
 \textit{s.t.}&   \begin{cases}
        \dot{x}_1 = x_2, \\
        \dot{x}_2 = -x_1 + u, \\
        |u| \le 1,\\
    x_1(0) = 2, \; x_2(0) = 0, \; x_2(T) = 0.
    \end{cases}
\end{split}
\end{equation*}
We reformulate the problem in Mayer form. Introduce the variable \(x_3 = t\):
\begin{equation*}
    \begin{cases}
        \dot{x}_1 = x_2, \\
        \dot{x}_2 = -x_1 + u, \\
        \dot{x}_3 = 1,\\
    \end{cases}
   \end{equation*}
with initial and terminal conditions:
$$x_1(0)=2,\; x_2(0)=0,\; x_3(0)=0,\; x_2(T)=0,$$
and the cost functional \(J = x_3(T) \to \min\).

A key advantage of the exact penalty formulation is that the 
differential constraints can be eliminated from the optimization 
entirely. In the numerical implementation, this is realized as 
follows: the state trajectory is obtained by fourth-order Runge--Kutta 
integration for a given control, so that \(\phi_{\mathrm{diff}} = 0\) 
up to the integration accuracy. This is justified by the exact penalty 
theory: since the minimum of the penalized functional is attained at a 
solution of the original problem, the differential constraints are 
automatically satisfied at the optimum, and the state trajectory can 
be computed by numerical integration for the current control at each 
iteration, without including the differential constraint in the 
penalized objective. Consequently, the penalty term \(\phi_{\mathrm{diff}}\) 
is not included in the numerical objective, and only the terminal 
penalty \(\rho |x_2(T)|\) is used. This is consistent with the exact 
penalty framework: on the admissible set \(\Omega\), we have 
\(\phi_{\mathrm{diff}} = 0\), and the exact penalty 
\(F_\lambda = J + \lambda \phi_{\mathrm{diff}} + \lambda \phi_{\mathrm{term}}\) 
reduces to \(J + \lambda \phi_{\mathrm{term}}\).

\subsection{Analytical Solution via the Maximum Principle}

The Hamiltonian is
\begin{equation*}
    H = \psi_1 x_2 + \psi_2(-x_1 + u) + \psi_3.
\end{equation*}
The adjoint system yields
\begin{equation*}
    \dot{\psi}_1 = \psi_2,\quad \dot{\psi}_2 = -\psi_1,\quad \dot{\psi}_3 = 0,
\end{equation*}
with general solution for $\psi_1$ and $\psi_2$ in the form
\begin{equation*}
    \psi_1 = A\sin(t+\varphi),\quad \psi_2 = A\cos(t+\varphi).
\end{equation*}
where $A$ and $\varphi$ are some constants.

According to Pontryagin's maximum principle, the Hamiltonian attains its maximum 
with respect to the control at the optimal solution. Hence,
\begin{equation*}
    u^*(t) = \operatorname{sign}(\psi_2(t)),
\end{equation*}
so \(u = \pm 1\). For \(u = 1\), the phase trajectory of the system is a circle 
centered at \((1,0)\); for \(u = -1\), it is a circle centered at \((-1,0)\). 
Therefore, there is at most one switching time \(\tau\), where \(\psi_2\) changes sign, 
whence
\[
\cos(\tau+\varphi) = 0.
\]

From the transversality conditions \eqref{eq:transversality_thm}, we obtain
\[
\psi_1(T) = 0, \qquad \psi_3(T) = -1,
\]
and for the free terminal time, the Hamiltonian vanishes (see \eqref{eq:hamiltonian_zero}):
\[
H(T) = 0.
\]
Thus, \(\psi_3 \equiv -1\) and, since \(\psi_1(t) = A\sin(t+\varphi)\), we have
$\sin(T+\varphi) = 0,$
and consequently
\begin{equation*}
    T - \tau = \frac{\pi}{2},
\end{equation*}
i.e., the time from the switching moment to the final point is \(\pi/2\).

\subsubsection{Construction of the Optimal Trajectory}

Collecting the information obtained from the analysis, we conclude (see Fig.  \ref{fig:phase}) that the 
optimal motion starts from the initial point \((2,0)\) with \(u = -1\) and 
follows a circle centered at \((-1,0)\) of radius \(3\). Upon reaching the 
point \((1,-\sqrt{5})\), the control switches to \(u = +1\) and the motion 
continues along a circle centered at \((1,0)\) passing through \((1,\pm\sqrt{5})\). 
The radius of this circle is \(\sqrt{5}\). After the switching, the system 
traverses a quarter of the circle (\(\pi/2\)) and reaches the terminal point 
\((1-\sqrt{5},0)\).

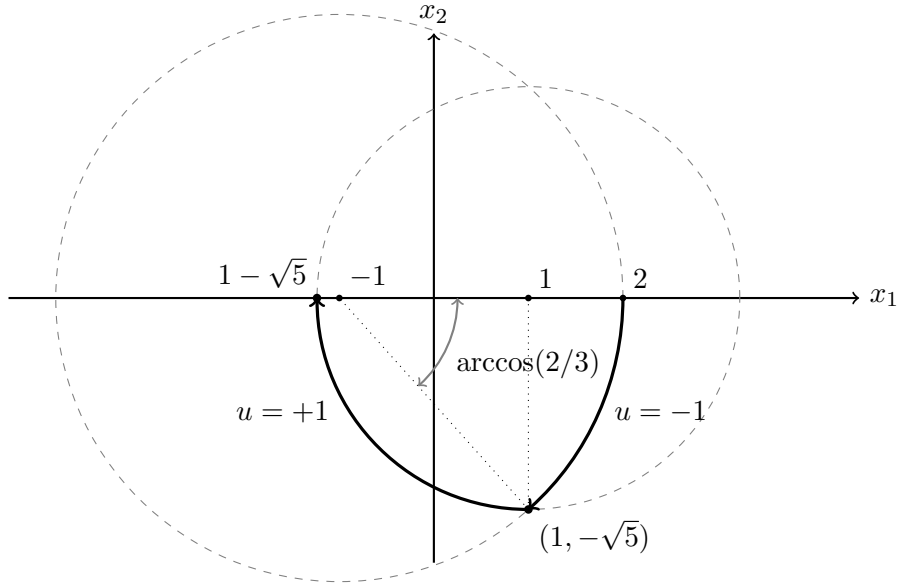
\begin{figure}[h!]
    \centering
\begin{tikzpicture}[scale=1.05]
    \draw[->, thick] (-4.5,0) -- (4.5,0) node[right] {$x_1$};
    \draw[->, thick] (0,-2.8) -- (0,2.8) node[above] {$x_2$};
    
    \fill (-1,0) circle (0.035) node[above right] {$-1$};
    \fill (1,0) circle (0.035) node[above right] {$1$};
    
    \draw[gray, dashed] (-1,0) circle (3);
    \draw[gray, dashed] (1,0) circle ({sqrt(5)});
    
    \coordinate (A) at (2,0);
    \coordinate (P) at (1,{-sqrt(5)});
    \coordinate (F) at ({1 - sqrt(5)},0);
    
    \fill (A) circle (0.035) node[above right] {$2$};
    \fill (P) circle (0.045) node[below right] {$(1,-\sqrt{5})$};
    \fill (F) circle (0.045) node[above left] {$1-\sqrt{5}$};
    
    \draw[very thick, black, ->]
        (2,0)
        arc[start angle=0, end angle={-acos(2/3)}, radius=3];
    
    \draw[very thick, black, ->]
        (P)
        arc[start angle=-90, end angle=-180, radius={sqrt(5)}];
    
    \node[black] at (2.4,-1.2) {$u=-1$};
    \node[black] at (-1.6,-1.2) {$u=+1$};
    
    \draw[dotted] (1,{-sqrt(5)}) -- (1,0);
    \draw[dotted] (1,{-sqrt(5)}) -- (-1,0);
    \draw[dotted] (1,0) -- ({1 - sqrt(5)},0);

 \draw[thick, gray, <->]
        (0.25,0)
        arc[start angle=0, end angle=-acos(2/3), radius=1.25];
    \node at (1.0,-0.7) {$\arccos(2/3)$};

\end{tikzpicture}
    \caption{Phase portrait of the optimal trajectory. The motion starts at 
             $(2,0)$, switches at $(1,-\sqrt{5})$, and ends at $(1-\sqrt{5},0)$.}
    \label{fig:phase}
\end{figure}

From geometry, the switching time is
\begin{equation*}
    \tau = \arccos\frac{2}{3}.
\end{equation*}
Consequently,
\begin{equation*}
    T = \frac{\pi}{2} + \arccos\frac{2}{3}.
\end{equation*}
Note that the solution of the system for \(u = -1\) is
\begin{equation*}
    \begin{cases}
        x_1(t) = 3\cos t - 1, \\
        x_2(t) = -3\sin t,
    \end{cases}
\end{equation*}
and therefore the motion occurs in the region \(x_2 \le 0\) (lower arcs of the circles).

From the condition \(H(T) = 0\), we obtain
\begin{equation*}
    \psi_2(T)(-x_1(T) + u(T)) - 1 = 0,
\end{equation*}
which yields \(\psi_2(T) = 1/\sqrt{5}\).

Thus, the optimal control has the form
\begin{equation*}
    u(t) =
    \begin{cases}
        -1, & t \in [0, \arccos(2/3)], \\
        +1, & t \in [\arccos(2/3), \arccos(2/3) + \pi/2].
    \end{cases}
\end{equation*}
The minimum time is
\begin{equation*}
    T = \arccos\frac{2}{3} + \frac{\pi}{2}.
\end{equation*}

\subsubsection{Verification of regularity}

For the harmonic oscillator problem, the terminal constraint is \(x_2(T) = 0\). 
We treat it as an equality constraint, so in the penalty formulation the 
terminal term is
\begin{equation*}
    \phi_{\mathrm{term}}(x,u) = |x_2(T)|.
\end{equation*}

For a single equality constraint, the classical regularity condition (linear 
independence of gradients) is trivially satisfied since the gradient 
\(\nabla\Phi^1 = (0,1)\) is nonzero. Consequently, for any 
point with \(\phi_{\mathrm{term}} > 0\) sufficiently close to the optimum, the 
subdifferential of \(\phi_{\mathrm{term}}\) is either \(\{1\}\) or \(\{-1\}\), and 
\(\operatorname{dist}(0,\partial\phi_{\mathrm{term}}) = 1\). Thus, USC holds with 
\(a = 1\).

\subsection{Numerical Solution}
The control \(u(t)\) is approximated by a piecewise constant function on a uniform grid with \(N = 200\) intervals. The state trajectory is obtained by numerically integrating the differential equation using the classical fourth-order Runge--Kutta method. Thus, the dynamical constraints \(\dot{x} = f(x,u)\) are satisfied by construction -- there is no need to penalize them. The optimization variables are only the control values \(u_k\) (\(k = 1,\dots,N\)) and the final time \(T\). The control bound \(|u| \le 1\) is handled by the parameterization \(u_k = \tanh(10\theta_k)\) with unconstrained \(\theta_k \in \mathbb{R}\). The objective functional reduces to
\[
J(T,\theta) = T + \rho |x_2(T)|,
\]
where \(\rho > 0\) is a penalty parameter (taken as \(\rho = 100\)) and the term \(|x_2(T)|\) enforces the terminal condition.

The resulting unconstrained optimization problem is solved using the 
Nel\-der--Mead simplex method (as implemented in \textsc{Matlab}'s 
\texttt{fminsearch} function). While more advanced algorithms exploiting essential nonsmooth properties 
exist (see, e.g., \cite{DemyanovAbbasov_2010}), the simple derivative-free 
Nelder--Mead method suffices for the purposes of this illustrative example: 
the objective functional is nonsmooth due to the absolute value term 
\(|x_2(T)|\) and the use of numerical integration, which makes derivative-free 
methods particularly suitable. Despite these nonsmooth features, the method converges reliably from a simple initial guess: the control is set to \(u(t) = -1\) on the first third of the interval and \(u(t) = +1\) thereafter (a plausible bang--bang structure with a switching point), and the final time is initialized as \(T = 3.5\).

The optimal time obtained is \(T = 2.41237\), which agrees with the analytical value \(T^* = 2.41186\) to within \(5\times10^{-4}\). The computed control exhibits the expected bang--bang structure with a single switching (Fig.~\ref{AM_fig:oscillator}), and the phase portrait follows the theoretical arcs of circles. These results confirm the effectiveness of the exact penalty approach.

\begin{figure}[h]
    \centering
    \includegraphics[width=0.72\textwidth]{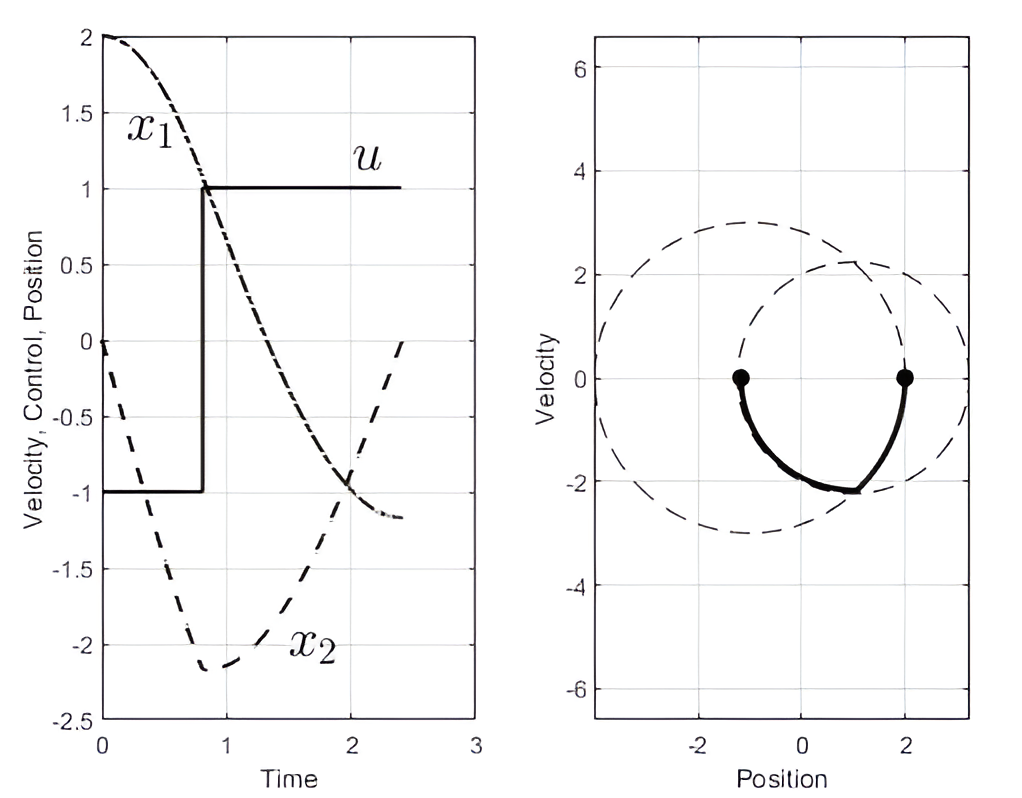}
    \caption{Optimal trajectory (position $x_1$ and velocity $x_2$) and control $u$ for the harmonic oscillator problem, obtained by the exact penalty method with $N = 200$, $T_{\text{opt}} = 2.41237$. The control exhibits a bang--bang structure with a single switching, and the velocity vanishes at the terminal time.}
    \label{AM_fig:oscillator}
\end{figure}

\section{Conclusions}
\label{sec:conclusions}

In this paper we have presented a unified derivation of transversality 
conditions in optimal control problems using exact penalty functions. 
The key regularity condition -- the Unified Separation Condition (USC) -- 
requires that the origin is uniformly separated from the subdifferential 
of the penalty function in a neighborhood of the admissible set. This 
condition reduces to LICQ for equality constraints and to MFCQ for 
inequality constraints in the smooth case, and remains applicable when 
the terminal constraints are nondifferentiable, where classical 
constraint qualifications are not defined.

The main contributions of this work are twofold. From the theoretical 
perspective, USC provides a single regularity framework that unifies 
and extends the classical LICQ and MFCQ conditions, and, assuming 
exactness of the penalty function, the transversality conditions for 
all major cases -- fixed and free terminal time, equality and inequality 
constraints, moving manifolds, and free left endpoint -- are derived 
using only classical variations and subdifferential calculus. This 
approach avoids the construction of cones of endpoint variations and 
the application of separation theorems, and it extends naturally to 
nonsmooth terminal constraints. From the computational perspective, 
the exact penalty formulation allows the differential constraints to 
be eliminated from the optimization entirely, so that only the cost 
functional and the terminal constraints are penalized. A numerical 
implementation for the time-optimal control of a harmonic oscillator 
confirms the practical applicability of the approach: the differential 
constraints are satisfied by construction via Runge--Kutta integration, 
and the computed optimal time agrees with the analytical value obtained 
via Pontryagin's maximum principle to within \(5 \times 10^{-4}\).

Future research directions include extending the approach to problems 
with phase constraints, where the USC must be adapted to handle 
state-space restrictions that are not purely terminal, and developing 
stable nonsmooth numerical algorithms for larger-scale optimal control 
problems.

\section*{Acknowledgments}
The author acknowledges the use of Perplexity AI and DeepSeek for language text editing and for assistance in writing the code for 
the numerical experiments. All scientific content, including the 
problem formulation, theoretical results, proofs, and analysis, 
is the original work of the authors. The author takes full 
responsibility for the integrity and accuracy of the entire 
manuscript.

\end{document}